\documentclass[12pt,reqno]{amsart}

\usepackage{amssymb}
\usepackage[all]{xy}
\usepackage{blkarray}
\usepackage{mathrsfs,calligra}
\usepackage{xcolor}
\usepackage[dvipdfmx]{graphicx}
\usepackage{here,float}

\usepackage{pgf}

\usepackage[vcentermath,enableskew]{youngtab}
\usepackage{ytableau}
\ytableausetup{centertableaux,nobaseline}


\newcommand{\version}{Ver.~0.0}
\newcommand{\setversion}[1]{\renewcommand{\version}{Ver.~{#1}}}
\setversion{0.1 [2022/04/06 17:02:49 JST]}
\setversion{0.2 [2022/05/11 22:18:15 JST]}
\setversion{0.5 [2022/06/12 14:19:34 JST]}
\setversion{0.9 [2022/06/21 18:49:30 JST]}
\setversion{0.95 [2022/09/08 23:17:01 JST]}
\setversion{0.96 [2022/09/19 23:39:23 JST]}
\setversion{1.00 [2022/09/21 16:34:19 JST]}

\title{Action of Hecke algebra on the double flag variety of type AIII}

\author{Lucas Fresse and Kyo Nishiyama}
\address{Universit\'e de Lorraine, CNRS, Institut \'Elie Cartan de Lorraine, UMR 7502, Vandoeu\-vre-l\`es-Nancy, F-54506, France}
\email{lucas.fresse@univ-lorraine.fr}
\address{Department of Mathematics, Aoyama Gakuin University, Fuchinobe 5-10-1, Chuo-ku, Sagamihara 252-5258, Japan}
\email{kyo@math.aoyama.ac.jp}
\thanks{K.~N.~is supported by JSPS KAKENHI Grant Number \#{21K03184}.}


\numberwithin{equation}{section}

\newtheorem{theorem}{Theorem}[section]
\newtheorem{lemma}[theorem]{Lemma}

\newtheorem{corollary}[theorem]{Corollary}

\theoremstyle{definition}
\newtheorem{example}[theorem]{Example}

\newcounter{penum}
\newenvironment{penumerate}{%
\begin{list}{$\;\;(\thepenum)$}{%
\usecounter{penum}
\setlength{\topsep}{0pt}
\setlength{\partopsep}{0pt}
\setlength{\parsep}{1ex}
\setlength{\itemindent}{0pt}
\setlength{\labelsep}{.5em}
\setlength{\labelwidth}{0pt}
\setlength{\leftmargin}{0pt}
\setlength{\rightmargin}{0pt}
\setlength{\itemsep}{0pt}
}}
{\end{list}}

\newcommand{\skipover}[1]{}

\newcommand{\Q}{\mathbb{Q}}

\newcommand{\bbG}{\mathbb{G}}
\newcommand{\C}{\mathbb{C}}
\newcommand{\bbF}{\mathbb{F}}
\newcommand{\bbP}{\mathbb{P}}

\newcommand{\Xfv}{\mathfrak{X}}
\newcommand{\Grass}{\mathrm{Gr}}

\newcommand{\GL}{\mathrm{GL}}
\newcommand{\PGL}{\mathrm{PGL}}

\newcommand{\diag}{\qopname\relax o{diag}}

\newcommand{\Ind}{\qopname\relax o{Ind}}
\newcommand{\trivial}{\mathbf{1}}

\newcommand{\Fun}{\qopname\relax o{Fun}}
\newcommand{\Stab}{\qopname\relax o{Stab}}

\newcommand{\Aut}{\qopname\relax o{Aut}}

\newcommand{\Vertices}{\mathcal{V}}

\newcommand{\lengthof}[1]{\ell(#1)}

\newcommand{\leftaction}{\mathrel{\raisebox{.4em}[0pt][0pt]{$\curvearrowright$}}}
\newcommand{\rightaction}{\,\raisebox{.4em}[0pt][0pt]{$\curvearrowleft$}\,}

\newcommand{\means}{{}\;\;$ \rightsquigarrow $\;\;{}}
\renewcommand{\Im}{\qopname\relax o{Im}}

\newcommand{\wgraph}[1]{\Gamma({#1})}
\newcommand{\Graphs}{\mathcal{G}}

\newcommand{\Heckeof}[1]{\Hecke({#1})}
\newcommand{\Hecke}{\mathscr{H}}

\newcommand{\vectwo}[2]{{\renewcommand{\arraystretch}{.85}\Bigl(\begin{array}{@{\,}c@{\,}}{#1}\\ {#2}\end{array}\Bigr)}}

\newcommand{\Flags}{\mathscr{F}\!\ell}

\newcommand{\lie}[1]{\mathfrak{#1}}

\newcommand{\wfrac}[2]{\dfrac{\,{#1}\,}{\,{#2}\,}}

\newcommand{\eb}{\boldsymbol{e}}

\newcommand{\scH}{\mathscr{H}}

\newcommand{\nopicture}[1]{}

\newcommand{\card}[1]{\# #1}
\newcommand{\rank}{\qopname\relax o{rank}}

\newcommand{\Mat}{\mathrm{M}}

\newcommand{\flag}{\mathcal{F}}

\newcommand{\conormal}{\mathcal{Y}}
\newcommand{\Xorbit}{\mathbb{O}}

\newcommand{\permutationsof}[1]{\mathfrak{S}_{#1}}
\newcommand{\symmetricgrpof}[1]{\mathfrak{S}_{#1}}
\newcommand{\ppermutations}{\mathfrak{T}}

\newcommand{\parameters}{\overline{\ppermutations}}

\newcommand{\q}{\boldsymbol{q}}
\newcommand{\ee}{\varepsilon}

\newcommand{\graphA}{\textcolor{black}{\mbox{\tiny $\begin{picture}(24,24)(0,-9)
\put(0,10){$\bullet$}\put(20,10){$\bullet$}
\put(0,-10){$\bullet$}\put(20,-10){$\bullet$}
\put(2,12){\line(1,-1){20}}
\put(22,12){\line(-1,-1){20}}
\end{picture}$}}}

\newcommand{\graphC}{\textcolor{black}{\mbox{\tiny $\begin{picture}(24,24)(0,-9)
\put(0,10){$\bullet$}\put(20,10){$\bullet$}
\put(0,-10){$\bullet$}\put(20,-10){$\bullet$}
\put(2,12){\line(0,-1){20}}
\put(22,12){\line(0,-1){20}}
\end{picture}$}}}

\newcommand{\graphB}{\textcolor{black}{\mbox{\tiny $\begin{picture}(24,24)(0,-9)
\put(0,10){$\bullet$}\put(20,10){$\bullet$}
\put(0,-10){$\bullet$}\put(20,-10){$\bullet$}
\put(2,12){\line(1,-1){20}}
\put(22,12){\circle{8}}
\end{picture}$}}}

\newcommand{\graphD}{\textcolor{black}{\mbox{\tiny $\begin{picture}(24,24)(0,-9)
\put(0,10){$\bullet$}\put(20,10){$\bullet$}
\put(0,-10){$\bullet$}\put(20,-10){$\bullet$}
\put(22,12){\line(-1,-1){20}}
\put(22,-8){\circle{8}}
\end{picture}$}}}

\newcommand{\graphE}{\textcolor{black}{\mbox{\tiny $\begin{picture}(24,24)(0,-9)
\put(0,10){$\bullet$}\put(20,10){$\bullet$}
\put(0,-10){$\bullet$}\put(20,-10){$\bullet$}
\put(2,12){\line(0,-1){20}}
\put(22,12){\circle{8}}
\end{picture}$}}}

\newcommand{\graphF}{\textcolor{black}{\mbox{\tiny $\begin{picture}(24,24)(0,-9)
\put(0,10){$\bullet$}\put(20,10){$\bullet$}
\put(0,-10){$\bullet$}\put(20,-10){$\bullet$}
\put(22,12){\line(0,-1){20}}
\put(2,12){\circle{8}}
\end{picture}$}}}

\newcommand{\graphG}{\textcolor{black}{\mbox{\tiny $\begin{picture}(24,24)(0,-9)
\put(0,10){$\bullet$}\put(20,10){$\bullet$}
\put(0,-10){$\bullet$}\put(20,-10){$\bullet$}
\put(22,12){\circle{8}}
\put(22,-8){\circle{8}}
\end{picture}$}}}

\newcommand{\graphH}{\textcolor{black}{\mbox{\tiny $\begin{picture}(24,24)(0,-9)
\put(0,10){$\bullet$}\put(20,10){$\bullet$}
\put(0,-10){$\bullet$}\put(20,-10){$\bullet$}
\put(22,12){\line(0,-1){20}}
\put(2,-8){\circle{8}}
\end{picture}$}}}

\newcommand{\graphI}{\textcolor{black}{\mbox{\tiny $\begin{picture}(24,24)(0,-9)
\put(0,10){$\bullet$}\put(20,10){$\bullet$}
\put(0,-10){$\bullet$}\put(20,-10){$\bullet$}
\put(2,12){\line(0,-1){20}}
\put(22,-8){\circle{8}}
\end{picture}$}}}

\newcommand{\graphJ}{\textcolor{black}{\mbox{\tiny $\begin{picture}(24,24)(0,-9)
\put(0,10){$\bullet$}\put(20,10){$\bullet$}
\put(0,-10){$\bullet$}\put(20,-10){$\bullet$}
\put(22,12){\line(-1,-1){20}}
\put(2,12){\circle{8}}
\end{picture}$}}}

\newcommand{\graphK}{\textcolor{black}{\mbox{\tiny $\begin{picture}(24,24)(0,-9)
\put(0,10){$\bullet$}\put(20,10){$\bullet$}
\put(0,-10){$\bullet$}\put(20,-10){$\bullet$}
\put(22,12){\circle{8}}
\put(2,-8){\circle{8}}
\end{picture}$}}}

\newcommand{\graphL}{\textcolor{black}{\mbox{\tiny $\begin{picture}(24,24)(0,-9)
\put(0,10){$\bullet$}\put(20,10){$\bullet$}
\put(0,-10){$\bullet$}\put(20,-10){$\bullet$}
\put(2,12){\circle{8}}
\put(22,-8){\circle{8}}
\end{picture}$}}}

\newcommand{\graphM}{\textcolor{black}{\mbox{\tiny $\begin{picture}(24,24)(0,-9)
\put(0,10){$\bullet$}\put(20,10){$\bullet$}
\put(0,-10){$\bullet$}\put(20,-10){$\bullet$}
\put(2,12){\line(1,-1){20}}
\put(2,-8){\circle{8}}
\end{picture}$}}}

\newcommand{\graphN}{\textcolor{black}{\mbox{\tiny $\begin{picture}(24,24)(0,-9)
\put(0,10){$\bullet$}\put(20,10){$\bullet$}
\put(0,-10){$\bullet$}\put(20,-10){$\bullet$}
\put(2,12){\circle{8}}
\put(22,12){\circle{8}}
\end{picture}$}}}

\newcommand{\graphO}{\textcolor{black}{\mbox{\tiny $\begin{picture}(24,24)(0,-9)
\put(0,10){$\bullet$}\put(20,10){$\bullet$}
\put(0,-10){$\bullet$}\put(20,-10){$\bullet$}
\put(2,12){\circle{8}}
\put(2,-8){\circle{8}}
\end{picture}$}}}

\newcommand{\graphP}{\textcolor{black}{\mbox{\tiny $\begin{picture}(24,24)(0,-9)
\put(0,10){$\bullet$}\put(20,10){$\bullet$}
\put(0,-10){$\bullet$}\put(20,-10){$\bullet$}
\put(2,-8){\circle{8}}
\put(22,-8){\circle{8}}
\end{picture}$}}}

\newcommand{\graphQ}{\textcolor{black}{\mbox{\tiny $\begin{picture}(24,24)(0,-9)
\put(0,10){$\bullet$}
\put(0,-10){$\bullet$}
\put(2,12){\line(0,-1){20}}
\end{picture}$}}}

\begin{document}

\begin{abstract}
Consider a connected reductive algebraic group $ G $ and a symmetric subgroup $ K $.  
Let $ \Xfv = K/B_K \times G/P $ be a double flag variety of finite type, 
where $ B_K $ is a Borel subgroup of $ K $, and $ P $ a parabolic subgroup of $ G $.  
A general argument shows that 
the orbit space $ \C\,\Xfv/K $ inherits a natural action of the Hecke algebra 
$ \Hecke = \Hecke(K, B_K) $ of double cosets via convolutions.  
However, it is a quite different problem to find out the explicit structure of the Hecke module.  

In this paper, for the double flag variety of type AIII, 
we determine the explicit action of $ \Hecke $ on $ \C\,\Xfv/K $ 
in a combinatorial way using graphs.    
As a by-product, we also get the description of the representation of the Weyl group on $ \C\,\Xfv/K $ 
as a direct sum of induced representations.
\end{abstract}

\maketitle

\setcounter{tocdepth}{1}

\tableofcontents

\section{Double flag varieties and Hecke algebra actions}

Let $ G $ be a connected reductive algebraic group 
with an involutive automorphism $ \theta $.  
We denote by $ K = G^{\theta} $ the subgroup of fixed points of $ \theta $ in $ G $.  
We assume $ K $ is connected for simplicity.  
Note that this assumption holds if $ G $ is simply connected.  

Let us consider 
a \emph{double flag variety} 
$ \Xfv = K/B_K \times G/P $, 
where $ B_K $ is a Borel subgroup of $ K $ and 
$ P $ is a parabolic subgroup of $ G $.  
We assume $ \Xfv $ is of finite type, i.e., there are finitely many orbits with respect to the 
diagonal $ K $ action on $ \Xfv $ (see \cite{NO.2011} and \cite{HNOO.2013}).  
Since 
$ \Xfv / K \simeq B_K \backslash G/ P $, 
this is equivalent to saying that there are finitely many $ B_K $-orbits on the partial flag variety $ G/P $, or
in other words, the natural action of $ K $ on $ G/P $ is spherical.  

Let us denote the Hecke algebra of $ (K, B_K) $ by 
$ \Hecke = \Heckeof{K, B_K} $.  
Then there exists a general recipe to define an action of $ \Hecke $ 
on the space of $ K $-orbits $ \C\,\Xfv/K $ in the double flag variety $ \Xfv $ by using the convolution product 
and the following double fibration maps (see \cite{Chriss.Ginzburg.1997}, for example).
\begin{eqnarray*}
\xymatrix @R-.3ex @M+.5ex @C-3ex @L+.5ex @H+1ex {
 & \ar[ld]_{p_{12}} \makebox[3ex][c]{$K/B_K {\times} K/B_K {\times} G/P$} \ar[rd]^{p_{23}} &
\\
\makebox[3ex][r]{$K/B_K {\times} K/B_K$}  & & \makebox[3ex][l]{$K/B_K {\times} G/P = \Xfv $}
}
\end{eqnarray*}
In this diagram, $ K $ acts diagonally, and all the maps respect the $ K $ action.  

However, in practice we prefer a simpler picture with the left $ B_K $ action:
\begin{eqnarray*}
\xymatrix @R-.7ex @M+.5ex @C-3ex @L+.5ex @H+1ex {
 & \ar[ld]_{p_{1}} \makebox[3ex][c]{$K \times_{B_K} G/P$} \ar[rd]^{p_{2}} &
\\
\makebox[3ex][r]{$K/B_K$}  & & \makebox[3ex][l]{$ G/P = X $}
}
\end{eqnarray*}
More generally, 
if $ X $ is a spherical $ K $-variety, 
Hecke algebra actions are considered by 
Mars-Springer \cite{Mars.Springer.1998} and Knop \cite{Knop.1997}.

Thus there exists an action of the Hecke algebra on 
the orbit space of $ \Xfv/ K \simeq B_K \backslash G / P $ so that 
the orbit space $ \C\,\Xfv/K $ is a Hecke module.  
However, there is no definite way to determine this module structure, and it seems difficult to describe the module structure even for a given explicit double flag variety.

Here in this paper, we will describe the explicit and concrete module structure of the Hecke algebra $ \scH $ 
for the case of the double flag variety of type AIII.  
The action is very explicit in terms of certain graphs, which represent $ K $-orbits.  
See Theorem \ref{thm:Hecke.algebra.action.generators}, which is the main theorem of this paper.  
From this theorem, 
we can also deduce the precise module structure of $ \C\,\Xfv/K $ 
as a representation of the Weyl group $ W_K $ of $ K $, 
which is isomorphic to $ \symmetricgrpof{p} \times \symmetricgrpof{q} $ in our situation.  
The representation is described in terms of a sum of induced representations.  
See Theorem \ref{thm:Weyl.group.representation}. 

To state the results in detail, let us first explain what is our double flag variety, and 
the structure of the orbit space.

\section{Double flag variety of type AIII}

In this section, the base field will be any field of characteristic other than $ 2 $.  
Later, we will consider the double flag varieties over finite fields.  

From now on, we concentrate on the case of the symmetric space of type AIII.  
\begin{itemize}
\item
$ G = \GL_n $ denotes the general linear group of order $ n $.  
\item
$ K = \GL_p \times \GL_q $ is a symmetric subgroup diagonally embedded into $ G $, where $ p + q = n $.  
\item
$ P = P_{(r, n - r)} $ denotes a standard maximal parabolic subgroup in $ G $ consisting of blockwise upper-triangular matrices with 
2 diagonal blocks of size $ r $ and $ n - r$.  
\item
$ B_K = B_p \times B_q $ is a Borel subgroup in $ K $, 
where $ B_p $ denotes the subgroup of $ \GL_p $ consisting of upper-triangular matrices.  
\end{itemize}
Thus we have 
\begin{equation*}
\begin{aligned}[t]
\Xfv &= K/B_K \times G/ P = \Bigl( \GL_p/B_p \times \GL_q/B_q \Bigr) \times \GL_n/P_{(r, n - r)} 
\\
&\simeq \Bigl( \Flags(V^+)\times \Flags(V^-)  \Bigr) \times \Grass_r(V),
\end{aligned}
\end{equation*}
where 
\begin{itemize}
\item $ V $ is an $ n $-dimensional vector space with a polar decomposition 
$V=V^+\oplus V^-$ and $ \dim V^+ = p$, $\dim V^- = q $.  
\item $\Grass_r(V)$ is the \emph{Grassmannian} of $r$-dimensional subspaces of $V$, and 
\item $\Flags(V^{\pm})$ denote the \emph{complete flag varieties} of $V^{\pm}$.  
\end{itemize}

It is not difficult to see 

\begin{lemma}
$ \# \Xfv / K < \infty $, i.e., $ \Xfv $ is of {finite type}.  
\end{lemma}

For general double flag varieties of finite type, we refer the readers to \cite{HNOO.2013}.

Write $ X = \Grass_r(V) \simeq G/P_{(r, n - r)} $, 
then $ K $ acts on $ X $ spherically, i.e., 
the action $ B_K \leftaction X $ has finitely many orbits.

\begin{lemma}
There is a natural bijection 
\begin{equation*}
\xymatrix @R -4.5ex @C -6ex @L 5pt
{ 
& \Xfv / K \ar[rr]^{\simeq} & \qquad \qquad & X / B_K
\\
& \text{\rotatebox[origin=c]{-90}{$\ni$}} &  & \text{\rotatebox[origin=c]{-90}{$\ni$}}
\\
& K \cdot ([\tau], \flag_0^+, \flag_0^-) \ar@{|->}[rr] & & B_K \cdot [\tau]}
\end{equation*}
where $ [\tau] \in \Grass_r(V) $, and $ \flag_0^{\pm} $ denote the standard flags of $ V^\pm $ stabilized by $ B_p $ and $ B_q $, respectively.  
\end{lemma}

In the following, {we will often identify} $ \Xfv / K $ and $ X / B_K $ via the above explicit bijection.

\section{Description of $ K $-orbits on $ \Xfv $}\label{section:description.Korbits}

Here we summarize the structure of the double flag variety and $ K $-orbits on it 
from our previous works.
For details we refer the readers to  
\cite{Fresse.N.2016,Fresse.N.2020,FN.arxive2021,2020ContempMath.FN}.

\subsection{Partial permutations}\label{section-2.1}

A \emph{partial permutation} of size $ p \times r $ is a matrix 
$ \tau_1 \in \Mat_{p, r} $ 
with entries in $ \{0,1\} $, in which the number of $ 1 $'s is less than or equal to $ 1 $ 
for any row and any column.  
(If $p=r$, we recover the set of partial permutation matrices considered in \cite{Fresse.N.2020}.)
Let us denote by $ \ppermutations_{p,r} $ the set of all the partial permutations in $ \Mat_{p, r} $.  
Put 
\begin{equation*}
\ppermutations=\ppermutations_{(p,q),r} 
:= \left\{ \tau=\begin{pmatrix} \tau_1 \\ \tau_2 \end{pmatrix} \in \ppermutations_{p,r} \times \ppermutations_{q,r} \Bigm| 
\rank \tau = r \right\} \subset \Mat_{p + q, r} , 
\end{equation*}
which is the set of pairs of partial permutations arranged vertically which are of {full rank}.  
Note that the symmetric group $ \permutationsof{r} $ of order $ r $ acts on this set from the right: 
$ \ppermutations \rightaction \permutationsof{r} $,
and we denote by 
$\parameters=\ppermutations/\permutationsof{r}$
the {quotient} by the symmetric group action.  

Let $ [\tau] := \Im \tau \in \Grass_r(V) $ denote the $ r $-dimensional subspace generated by the column vectors of $ \tau $.  

\begin{theorem}[{\cite[Theorem 2.2]{FN.arxive2021}}]
The map $ \ppermutations \ni \tau \mapsto [\tau] \in \Grass_r(V) $ 
factors through to a bijection
\begin{equation*}
\parameters=\ppermutations/ \permutationsof{r} \xrightarrow{\;\;\simeq\;\;}  X / B_K \simeq \Xfv / K 
\end{equation*}
so that we get the parametrization of the $ K $-orbits in the double flag variety 
$ \Xfv / K \simeq \parameters $.  
\end{theorem}

If there is no confusion, we will identify  
a matrix $ \tau \in \ppermutations $ with  
its representative in $ \parameters $.  
Thus $ \tau $ often represents a $ K $-orbit in $ \Xfv $.  
Note also that the Weyl group 
$ W_K = \symmetricgrpof{p} \times \symmetricgrpof{q} $ acts on $ \parameters $ on the left in a natural way.

\subsection{Graphs}

There exists a convenient presentation of $\tau\in\parameters$ by using graphs.  
Let us explain it.

For 
$\tau\in\parameters$, 
we consider a {graph} $\wgraph{\tau}$ determined by the following rule.
\begin{itemize}
\item
It has two kinds of \emph{vertices}: ``positive'' vertices $ \Vertices^+_p = \{ 1^+,\ldots,p^+ \} $ and 
``negative'' vertices $ \Vertices^-_q = \{ 1^-,\ldots,q^- \} $, 
both being displayed along two horizontal lines. 
\item
Draw \emph{edges} between $i^+ \in \Vertices^+_p $ and $j^- \in \Vertices^-_q$ if $ \tau $ contains 
\emph{two} $ 1 $'s {in the same column}, at rows $ i^+ $ and $ j^- $.  
\item
There are marked vertices: mark the vertex $i^+$ (or $j^-$) if $\tau$ contains \emph{only one} $ 1 $ at 
row $i^+$ (or $j^-$) in a column.
\item
As a result we get 
$ \# (\text{edges}) + \# (\text{marked vertices}) = r $.
\end{itemize}

\begin{example}
To understand the graphs, let us give an example.
When $(p,q)=(5,3)$ and $r=4$, we get
\begin{equation}
\label{2.1}
\tau=\mbox{\tiny $\begin{pmatrix} 0 & 0 & 0 & 0 \\ 1 & 0 & 0 & 0 \\ 0 & 0 & 0 & 0 \\ 0 & 1 & 0 & 0 \\ 0 & 0 & 1 & 0 \\ \hline 0 & 1 & 0 & 0 \\ 0 & 0 & 0 & 1 \\ 1 & 0 & 0 & 0 \end{pmatrix}$}
\qquad
\mapsto
\qquad
\wgraph{\tau}=
\mbox{\tiny
$\begin{picture}(100,25)(0,0)
\put(0,11){$\bullet$}\put(20,11){$\bullet$}\put(40,11){$\bullet$}\put(60,11){$\bullet$}\put(80,11){$\bullet$}
\put(0,20){$1^+$}\put(20,20){$2^+$}\put(40,20){$3^+$}\put(60,20){$4^+$}\put(80,20){$5^+$}
\put(0,-11){$\bullet$}\put(20,-11){$\bullet$}\put(40,-11){$\bullet$}
\put(0,-22){$1^-$}\put(20,-22){$2^-$}\put(40,-22){$3^-$}
\put(21,13){\line(1,-1){21}}
\put(61,12){\line(-3,-1){60}}
\put(82,13){\circle{8}}
\put(22,-9){\circle{8}}
\end{picture}$}
\end{equation}
\end{example}


The set of graphs of this type will be denoted by 
$ \Graphs((p,q), r) = \{ \wgraph{\tau} \mid \tau\in\parameters \} $.  
The graphs are characterized by the properties listed above.
Note that 
every vertex is incident with at most one edge or mark, 
and that there is no edge joining two distinct vertices of the same sign.

We summarize the description of orbits using graphs into the following lemma.

\begin{lemma}
The graphs classify $ K $-orbits in $ \Xfv $.
\begin{equation*}
\xymatrix @R -4.5ex @C -6ex @L 5pt
{ 
\Xfv / K \simeq X/B_K \ar@{<-}[rr]^(.7){\simeq} & \qquad \qquad
& \parameters \ar[rr]^(.3){\simeq} & \qquad \qquad & \Graphs((p,q), r)
\\
 \text{\rotatebox[origin=c]{-90}{$\ni$}} &  & \text{\rotatebox[origin=c]{-90}{$\ni$}} &  & \text{\rotatebox[origin=c]{-90}{$\ni$}}
\\
B_K \cdot [\tau] \ar@{<-|}[rr] & & \tau \ar@{|->}[rr] & & \wgraph{\tau}}
\end{equation*}
%
\end{lemma}

\subsection{Orbital invariants: $a^{\pm}(\tau)$, $b(\tau)$, $c(\tau)$ and $R(\tau)=(r_{i,j}(\tau))$}\;\;
\label{S3.3}

For the graph $\wgraph{\tau}$ we define:
\begin{itemize}
\item 
We set the \emph{degree} of vertices as $ \deg i^{\pm} := 0, 1, 2 $, 
depending on whether it is not incident with an edge nor marked, 
the end point of an edge, or marked, respectively.
\item
$a^{\pm}(\tau) := \# \{ (i^{\pm},j^{\pm}) \mid i<j \text{ and  } \deg(i^{\pm})<\deg(j^{\pm}) \} $

\item
$b(\tau) := \# \{ \text{edges} \} $

\item
$c(\tau) := \# \{ \text{\emph{crossings} of edges} \} $, 
i.e., the number of pairs of edges $(i^+,j^-)$ and $(k^+,\ell^-)$ such that $i<k$ and $j>\ell$.

\item 
$r_{i,j}(\tau) := \# (\text{edges}) + \# (\text{marked vertices}) $ 
with vertices among $\{1^+,\ldots,i^+\}\times\{1^-,\ldots,j^-\}$.  
\item $R(\tau) := \left( r_{i,j}(\tau) \right)_{0\leq i\leq p,\,0\leq j\leq q} \in \Mat_{p+1, q+1} $ : the ``\emph{rank matrix}''.  
\end{itemize}

We need $a^\pm(\tau)$, $b(\tau)$, $c(\tau)$ to give a dimension formula for the $K$-orbits in $\Xfv$ below, 
while the matrices $R(\tau)$ are to be used to describe the closure relations of orbits.

We also define a decomposition 
\begin{equation*}
\Vertices^+_p = \{1,\ldots,p\}=I\sqcup L\sqcup L',
\end{equation*}
where $I$ (resp. $L$, resp. $L'$) denotes the set of elements $i\in\{1,\ldots,p\}$ such that $i^+$ is a vertex of $\wgraph{\tau}$ of degree $1$
(resp. $2$, resp. $0$).  

A decomposition 
\begin{equation*}
\Vertices^-_q = \{1,\ldots,q\}=J\sqcup M\sqcup M'
\end{equation*}
is defined similarly.  
Namely, 
$J$ (resp. $M$, resp. $M'$) consists of the elements $j$ such that $j^-$ has degree $1$ (resp. $2$, resp. $0$)

Let $\sigma:J\to I$ be the \emph{bijection} defined by $\sigma(j)=i$ if $(i^+,j^-)$ is an edge in $\wgraph{\tau}$.

Note that $\tau$ is characterized by the subsets $I$, $L$, $L'$, $J$, $M$, $M'$ and the bijection $\sigma:J\to I$.
Also note that we have $b(\tau)=\card{I}=\card{J}$, and $c(\tau)$ is the number of inversions of $\sigma$.

\begin{example}\label{E2.1}\;\;
Let $ \eb_i^{\pm} $ be a standard basis vector of $ V^{\pm} $.  
For $\tau$ as in (\ref{2.1}), the associated graph is given as
\begin{gather*}
\tau = \vectwo{\tau_1}{\tau_2} = 
\left( \begin{array}{cccc} 
\eb_2^+ & \eb_4^+ & \eb_5^+ & 0 \\ \hline
\eb_3^- & \eb_1^- & 0 & \eb_2^- 
\end{array} \right)
\text{\means }
\wgraph{\tau}=
\mbox{\tiny
$\begin{picture}(100,25)(0,0)
\put(0,11){$\bullet$}\put(20,11){$\bullet$}\put(40,11){$\bullet$}\put(60,11){$\bullet$}\put(80,11){$\bullet$}
\put(0,20){$1^+$}\put(20,20){$2^+$}\put(40,20){$3^+$}\put(60,20){$4^+$}\put(80,20){$5^+$}
\put(0,-11){$\bullet$}\put(20,-11){$\bullet$}\put(40,-11){$\bullet$}
\put(0,-22){$1^-$}\put(20,-22){$2^-$}\put(40,-22){$3^-$}
\put(21,13){\line(1,-1){21}}
\put(61,12){\line(-3,-1){60}}
\put(82,13){\circle{8}}
\put(22,-9){\circle{8}}
\end{picture}$}
\qquad \text{ then}
\\
\begin{aligned}
 & a^+(\tau)=7,\quad a^-(\tau)=1,\quad b(\tau)=2,\quad c(\tau)=1, \quad
R(\tau)=\mbox{\tiny $\begin{pmatrix} 0 & 0 & 1 & 1 \\ 0 & 0 & 1 & 1 \\ 0 & 0 & 1 & 2 \\ 0 & 0 & 1 & 2 \\ 0 & 1 & 2 & 3 \\ 1 & 2 & 3 & 4 \end{pmatrix}$} \\
& 
\begin{array}{llll}
I=\{2,4\}, & L=\{5\}, & L'=\{1,3\},\\[2mm]
J=\{1,3\}, & M=\{2\}, & M'=\emptyset,
\quad\sigma=\begin{pmatrix} 1 & 3 \\ 4 & 2 \end{pmatrix}\in\mathrm{Bij}(J,I)
\end{array} \\
&
[\tau]=\langle \eb_2^++\eb_3^-,\eb_4^++\eb_1^-,\eb_5^+,\eb_2^-\rangle.
\end{aligned}
\end{gather*}
\end{example}

\subsection{Dimensions and closure relations of orbits}\label{section-2.2}

\vskip 1ex

Recall the base point $([\tau],\flag_0^+,\flag_0^-)$ in $\Xfv=\Grass_r(V)\times\Flags(V^+)\times\Flags(V^-)$.

\begin{theorem}[{\cite[Theorem 2.2]{FN.arxive2021}}]\label{T1}
Denote a $ K $-orbit in $\Xfv$ by 
$\Xorbit_\tau:=K\cdot([\tau],\flag_0^+,\flag_0^-)$.  
%
\begin{enumerate}
\item $\dim\Xorbit_\tau=\wfrac{p(p-1)}{2}+\wfrac{q(q-1)}{2}+a^+(\tau)+a^-(\tau)+\wfrac{b(\tau)(b(\tau)+1)}{2}+c(\tau)$.

\item 
$
\begin{aligned}[t]
\Xorbit_\tau = \{ (W,\flag^+, 
\flag^-)
\mid 
&
\dim W\cap(\flag_i^+ + \flag_j^-)=r_{i,j}(\tau) \;\; 
\\
&
\qquad\qquad
\text{ for any }  (i,j)\in\{0,\ldots,p\}\times\{0,\ldots,q\} \} .
\end{aligned}
$

\item $\overline{\Xorbit_\tau}\subset\overline{\Xorbit_{\tau'}} 
\iff r_{i,j}(\tau)\geq r_{i,j}(\tau') \;\; \text{ for any }  (i,j)\in\{0,\ldots,p\}\times\{0,\ldots,q\}$.
\end{enumerate}
\end{theorem}

We can describe the cover relation of the closure of orbits, which is not presented here 
(see {\cite[Theorem 2.3]{FN.arxive2021}}).  
Taking this for granted, we have

\begin{corollary}
If $\Xorbit_{\tau'}$ covers $\Xorbit_\tau$ then $\dim\Xorbit_{\tau'}=\dim\Xorbit_\tau+1$ holds.
\end{corollary}

\begin{figure}[H]
\caption{Closure relations of $ K $-orbits for $ p = q = r = 2 $}\label{figure2}
\textcolor{blue}{\xymatrixcolsep{1pc}
\xymatrix{
\textcolor{darkgray}{\dim:6}
& & & & & \graphA \ar@/_/@{-}[lld] \ar@{-}[d] \ar@/^/@{-}[rrd]  & & & & & \\
\textcolor{darkgray}{5} & & & \graphB \ar@{-}[lld] \ar@{-}[d] \ar@{-}[rrd] & & \graphC \ar@/_/@{-}[lllld] \ar@{-}[lld] \ar@{-}[rrd] \ar@/^/@{-}[rrrrd]  & & \graphD \ar@{-}[lld] \ar@{-}[d] \ar@{-}[rrd] & & & \\
\textcolor{darkgray}{4} & \graphE \ar@{-}[rd]\ar@/_/@{-}[rrrd] & & \graphF \ar@{-}[ld]\ar@/_/@{-}[rrrd] & & \graphG \ar@/_/@{-}[ld]\ar@/^/@{-}[rd] & & \graphH \ar@{-}[rd]\ar@/^/@{-}[llld] & & \graphI \ar@{-}[ld]\ar@/^/@{-}[llld] & \\
\textcolor{darkgray}{3} & & \graphJ \ar@{-}[rd]\ar@/_/@{-}[rrrd] & & \graphK \ar@{-}[rd] & & \graphL \ar@{-}[ld] & & \graphM \ar@{-}[ld]\ar@/^/@{-}[llld] & & \\
\textcolor{darkgray}{2} & & & \graphN & & \graphO & & \graphP & & &}}
\end{figure}
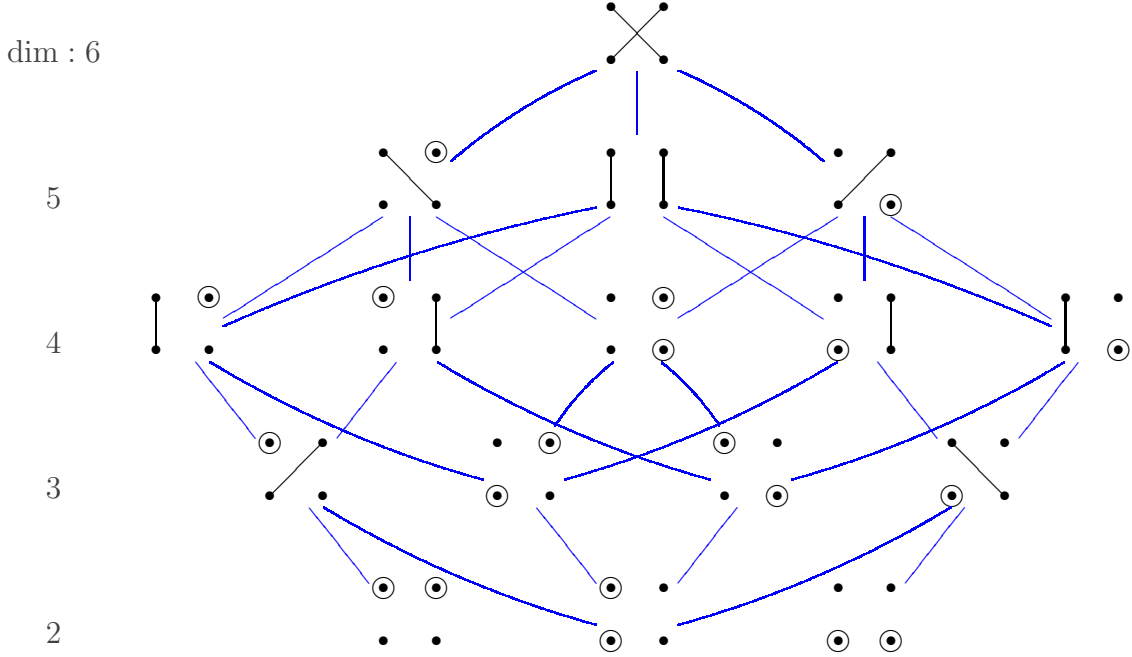

\subsection{The number of orbits}\label{subsection:number.of.orbits}

Let $ (k, s, t) $ be nonnegative integers which satisfy 
\begin{gather*}
p \geq k + s, \quad q \geq k + t , \quad r = k+s+t .
\end{gather*}
Put $ s'=p-k-s\;  \text{ and  } \;  t'=q-k-t $.  
Consider the subgroup $ H_{k,s,t} \subset \permutationsof{p} {\times} \permutationsof{q}$ defined by 
\begin{eqnarray*}
H_{k,s,t} & = & \{(\text{\emph{$a_1$}},a_2,a_3; \text{\emph{$a_1 $}},b_2,b_3) 
\in (\permutationsof{k} {\times} \permutationsof{s} {\times}  \permutationsof{s'}) 
{\times} (\permutationsof{k} {\times} \permutationsof{t} {\times}  \permutationsof{t'}) \} \\
 & \cong & \Delta\permutationsof{k} {\times}  \permutationsof{s} {\times}  \permutationsof{s'} {\times}  \permutationsof{t} {\times}  \permutationsof{t'},
\end{eqnarray*}
where $\Delta\permutationsof{k} \subset \permutationsof{k}^2 $ stands for the diagonal subgroup.  

\begin{theorem}[{\cite[Corollary 2.13]{FN.arxive2021}}]\label{thm:number.orbits}
The total number of $K$-orbits in $\Xfv$ is given by
\[\#\Xfv/K=\sum_{(k,s,t)}\dim \Ind_{H_{k,s,t}}^{\permutationsof{p}\times\permutationsof{q}} \trivial
= \sum_{(k,s,t)}\binom{p}{k,s,s'}\binom{q}{k,t,t'}k!,\]
where the sums are running over triples $(k,s,t)$ 
as above.
\end{theorem}

\section{Setting over finite fields}\label{section:setting.finite.field}

Based on the classification of orbits, we will calculate the Hecke algebra action on the orbit space.  
For this, we follow the classical recipe of Iwahori \cite{Iwahori.1964}, 
and we will consider everything over the finite field $ \bbF = \bbF_{\q} $ of $ \q $-elements from now on.  
(The letter $ q $ is already used to denote the size of the second block for $ K $.  
But the number of elements of a finite field is customary denoted also by the letter ``$ q $''.  
To distinguish them, we will use $ \q $ instead of $ q $ for the finite field $ \bbF_{\q} $.)

Summary of the notation over the finite fields:
\begin{equation*}
\begin{array}{c|c|lcccccccccc}
G & \GL_n(\bbF) \\
K & \GL_p(\bbF) \times \GL_q(\bbF) & \text{a symmetric subgroup of $ G $} \\
B_K & B_p(\bbF) \times B_q(\bbF) & \text{a Borel subgroup of $ K $} \\
W_K & \permutationsof{p} \times \permutationsof{q} & \text{the Weyl group of $ K $}\\
G/P & \GL_n(\bbF)/P_{(r, n - r)}(\bbF) \simeq \Grass_r(\bbF^n) & \text{Grassmannian of $ r $-spaces in $ \bbF^n $} \\
B_K \backslash G/P & \parameters = \ppermutations_{(p, q; r)}/ \permutationsof{r} & \text{the space of partial permutations} 
\end{array}
\end{equation*}
%

\skipover{
For the sake of our reference, I will give the numbers of elements in $ K $ and $ B_K $.  
Actually, we really do not need them.

\begin{lemma}
If we put $ \q = \# \bbF $, 
\begin{equation*}
\# \GL_n = \q^{n (n - 1)/2} \prod_{k = 1}^n (\q^k - 1) , \qquad
\# B_n = (\q - 1)^n \q^{n (n - 1)/2}
\end{equation*}
\end{lemma}

\begin{proof}
This is well-known.  I will give an account only for my convenience.  

Note that $ \GL_n / P_{(1, n - 1)} \simeq \bbP^{n - 1} $ is a projective space of dimension $ n - 1 $, 
and it is easy to see that 
\begin{equation*}
\# \bbP^{n - 1} = \q^{n - 1} + \q^{n - 2} + \dots + \q + 1 = \wfrac{\q^n - 1}{\q - 1} =: \q^{[k]}
\end{equation*}
Denote $ a_n = \# \GL_n $.  
Since $ P_{(1, n - 1)} = (\GL_1 \times \GL_{n - 1}) \ltimes \bbF^{n - 1} $, 
we get 
\begin{equation*}
\# P_{(1, n - 1)} = \# \bbF^{\times} \cdot a_{n - 1} \cdot (\# \bbF)^{n - 1} 
= (\q - 1) \q^{n - 1} a_{n - 1} .
\end{equation*}
Thus we get 
\begin{equation*}
a_n = \# \bbP^{n - 1} \times \# P_{(1, n - 1)} 
= \wfrac{\q^n - 1}{\q - 1} \cdot (\q - 1) \q^{n - 1} a_{n - 1} 
= (\q^n - 1) \q^{n - 1} a_{n - 1}.
\end{equation*}
Hence the formula for $ \# \GL_n $.  
The formula for $ \# B_n $ is easy.
\end{proof}

From this lemma, it is easy to derive the formula for $ \# K $ and $ \# B_K $, which I won't give here.
}

In addition to this, we also use the following notation.  
\begin{itemize}
\item
$ s_i = (i, i+1) $: simple reflection (a transposition in $ W_K $), and 
\emph{$ T_i = T_{s_i} $} is the corresponding generator in the Hecke algebra 
$ \Hecke = \Hecke(K, B_K) $.

\item
Recall pairs of partial permutations 
$ \tau = \vectwo{\tau_1}{\tau_2} \in \parameters=\ppermutations/\permutationsof{r}$ 
of full rank $ r $.  
The matrix $ \tau $ is \emph{identified} with its image $ [\tau] \in X=\Grass_r(\mathbb{F}^{n}) $ 
(thus \emph{we often omit $ [ \;\; ] $} below).  

\item
Let $ \Xorbit_{\tau} = B_K\cdot\tau$ be a $B_K$-orbit in the Grassmannian $ X $.  
Then {$ \xi_{\tau} $} denotes {the characteristic function} of the orbit $ \Xorbit_{\tau} $.  
\end{itemize}

We are interested in the action of $ T_i $, 
$ T_i \ast \xi_{\tau} $ for $ \tau \in \ppermutations $.  
To calculate it, we recall some basic facts on the action of Hecke algebras.

\section{Hecke algebra of double cosets}\label{section:Hecke.alg.2.cosets}

In this section, we consider a general finite group 
and review some general properties of a Hecke algebra of double cosets.  
For that reason, we will denote by $ K $ a general finite group.  
This notation is effective only in this section, 
but there is no harm to consider it as the already defined $ K $ (over a finite field) above.  

Let us take a subgroup $ B \subset K $ (again $ B $ does not necessarily mean a Borel subgroup) 
and consider the convolution algebra of 
$ B $-spherical functions on $ K $.  
Note that these functions are $ \C $-valued functions.  
This algebra is called the Hecke algebra of double cosets and 
we denote it as $ \scH = \scH(K, B) $.  
Namely, 
\begin{equation*}
\scH(K, B) = \{ f : K \to \C \mid 
f(h_1 k h_2) = f(k) \text{ for } h_1, h_2 \in B \text{ and } k \in K \}
\end{equation*}
and the convolution product is defined by 
\begin{equation*}
a \ast b (x) = \int_K a(k) b(k^{-1} x) dk = \wfrac{1}{\# K} \sum_{k \in K} a(k) b(k^{-1} x)
\end{equation*}
where the integral $ \int_K dk $ is taken with respect to the normalized Haar measure of the finite group $ K $.  
As written above, the integral is just the pointwise sum divided by the whole volume $ \# K $, 
but we prefer the notation using integral $ \int_K $.  

Put $ W = B \backslash K / B = \{ w \} $,  
identified with the set of representatives in $ K $, which we pick and fix once and for all.  
Let us consider characteristic functions on the double cosets $ B w B \subset K $ so that they form a basis of $ \scH $.  
However, since we would like to get an identity element for the double coset $ B = B e B $, 
we will normalize the characteristic functions by $ \#K/\#B $.  
Thus we put 
\begin{equation*}
T_w = \wfrac{\# K}{\# B} \cdot \trivial_{B w B} 
\qquad
(w \in W).
\end{equation*}
Then 
$ \{ T_w \}_{w \in W} $ forms a basis of $ \scH $ over $ \C $.  

Let $ X $ be a finite set and 
assume $ K $ acts on $ X $ from the left.   
We consider 
the space of functions 
$ \Fun^B(X) $ on $ X $ which are $ B $-invariant.  
The Hecke algebra $ \scH $ acts on $ \Fun^B(X) $ via the convolution again:
\begin{equation*}
f \ast \xi (x) := \int_K f(k) \xi(k^{-1} x) dk 
\qquad
(f \in \scH, \;\; \xi \in \Fun^B(X)).
\end{equation*}

We denote by $ \ppermutations = X/B $ identified with the set of representatives in $ X $.

We denote by $ \xi_{\tau} $ the characteristic function on a $ B $-orbit 
$ B \tau \subset X $ so that 
$ \{ \xi_{\tau} \}_{\tau \in \ppermutations} $ is a basis of $ \Fun^B(X) $.  

Let us calculate the convolution:
\begin{align}
T_w \ast \xi_{\tau}(x) 
&= \int_K T_w(g) \xi_{\tau}(g^{-1} x) dg \notag \\
&= \wfrac{1}{\# K} \sum_{g \in K} T_w(g) \xi_{\tau}(g^{-1} x) \notag \\
&= \wfrac{1}{\# K} \sum_{g \in B w B} T_w(g) \xi_{\tau}(g^{-1} x).
\label{eq:fw.convoluted.with.xitau}
\end{align}
For this sum, 
only $ g \in B w B \cap a K_{\tau} B $ contributes, 
where $ a \in K $ is chosen as $ x = a \tau $
(we assume that $x\in K\tau$, otherwise the sum is zero), and $ K_{\tau} = \Stab_K(\tau) $ denotes the stabilizer of $ \tau \in X $.    
In fact, 
$ \xi_{\tau}(g^{-1} x) \neq 0 $ iff $ g^{-1} x \in B \tau $.  
Since $ x = a \tau $, we get 
\begin{equation*}
g^{-1} x = b \tau
\iff g^{-1} a \tau = b \tau
\iff \tau = a^{-1} g b \tau
\end{equation*}
which means $ a^{-1} g b \in K_{\tau} = \Stab_K(\tau) $.  
Thus we get $ g \in a K_{\tau} B $.

Since 
$ T_w \ast \xi_{\tau} \in \Fun^B(X) $, it is a linear combination of various 
$ \xi_{\tau_i} $'s for $ \tau_i \in \ppermutations $.  
From the last expression \eqref{eq:fw.convoluted.with.xitau}, 
it is easy to see that if $ g \in B w B $ contributes to the sum nontrivially then 
\begin{equation*}
g^{-1} x \in B \tau \iff x \in g B \tau \iff B x \subset B w B \tau.
\end{equation*}
Let us decompose 
\begin{equation}\label{eq:BwB.tau}
B w B \tau = \sqcup_{i = 1}^N B \tau_i
\end{equation}
Thus we only have to consider the cases $ x = \tau_i \; (1 \leq i \leq N) $.  
If we choose $ a_i $'s which satisfy $ \tau_i = a_i \tau $, then 
the above consideration tells us 

\begin{theorem}\label{thm:action.of.fw.generalsetting}
The Hecke operator $ T_w \;\; (w \in W) $ acts on $ \Fun^B(X) $ by 
\begin{equation*}
T_w \ast \xi_{\tau} = \sum_{i = 1}^N \wfrac{\# (a_i K_{\tau} B \cap B w B)}{\# B} \xi_{\tau_i} .
\end{equation*}
\end{theorem}

\begin{proof}
This theorem follows from the discussion above.  
Note that $ T_w $ is normalized by the constant $ \# K / \# B $.
\end{proof}

We shall apply this formula to our situation.  
(It's still interesting to consider vector bundle case in general.  We postpone it as future study.)

\section{Double cosets multiplications}\label{section:2.cosets.multiplication}

Let us return to the setting of \S\ref{section:setting.finite.field}, 
however we make the assumption that the ground field $ \bbF $ is algebraically closed, of characteristic $\neq 2 $, 
which will take place only in this section.

Let $ s_i = (i, i+1) $ be a simple reflection (a transposition in $ W_K $), and 
put $ T_i = T_{s_i} $ be the corresponding element in the Hecke algebra.  
We are interested in $ T_i \ast \xi_{\tau} $ for $ \tau \in \ppermutations $.  
As in \S\ref{section:description.Korbits},  
$ \tau \in \ppermutations $ is often identified with a graph with two subsets of vertices 
$ \Vertices^+_p $ and $ \Vertices^-_q $ (of $ p $ and $ q $ elements respectively) which are equipped with 
several edges and marked vertices.  
Recall that $ W_K \simeq \symmetricgrpof{p} \times \symmetricgrpof{q} $ 
acts on $ \ppermutations $ by the matrix multiplication from the left, which descends to the action on 
$ \parameters $, the set of parameters of orbits.  
This action can be identified with the natural action of $ W_K $ on the graphs, induced by that on the vertices.  

The following key lemma corresponds to Equation \eqref{eq:BwB.tau} in the present situation.

\begin{lemma}\label{key-lemma:2coset.multiplication}
A double coset $ B_K s_i B_K $ generates at most two $ B_K $-orbits on the Grassmannian 
$ X = \Grass_r(\bbF^n) $.  Namely we have 
\begin{equation*}
B_K s_i B_K \cdot \tau = \begin{cases}
B_K s_i \tau = B_K \tau & \text{if $ s_i \tau = \tau $} 
\hfill \makebox[0pt][l]{\text{\upshape{}case (I)}}
\\
B_K s_i \tau \cup B_K \tau & \text{if $ s_i \tau \neq \tau $ and $ \tau $ is among $ (\ast) $}\qquad
\hfill \makebox[0pt][l]{\text{\upshape{}case (II)}}
\\
B_K s_i \tau & \text{if $ s_i \tau \neq \tau $ and $ \tau $ is among $ (\ast\ast) $}
\hfill \makebox[0pt][l]{\text{\upshape{}case (III)}}
\end{cases}
\end{equation*}
where 
$ (\ast) $ denotes the case of {\upshape{}(1), (3), (6), (8)} in Table \ref{figure:alpha.component.stabilizer.Borel} 
in Appendix {\upshape\S\ref{appendix:calculation.stabilzer}}, and 
$ (\ast\ast) $ denotes the case of {\upshape{}(2), (4), (5), (7)} (ibid.).
\end{lemma}

\begin{proof}
Let $B_K\cdot\tau$ be a $B_K$-orbit of the Grassmannian $ X=\mathrm{Gr}_r(\mathbb{F}^n)$.
Let $i\in\{1,\ldots,p-1\}$ and let $P_i=B_K\sqcup B_K s_i B_K$ be the corresponding minimal parabolic subgroup ($s_i$ is the corresponding simple reflection).

In Appendix \S\ref{appendix:calculation.stabilzer}, 
we compute the isotropy subgroup $P_i^\tau:=\{g\in P_i:g\cdot\tau=\tau\}\subset P_i$. More precisely, let $U_i$ be the unipotent radical of $P_i$ and let $L_i$ be the standard Levi subgroup of $P_i$. The quotient $L_i/Z(L_i)$ is isomorphic to $\PGL_2(\mathbb{F})$. By considering the Levi decomposition $P_i=L_i\ltimes U_i$, we get a morphism of groups
\[\pi_i:P_i\to L_i\to L_i/Z(L_i)\cong \PGL_2(\mathbb{F})\]
and a morphism of Lie algebras
\[d\pi_i:\mathfrak{p}_i=\mathrm{Lie}(P_i)\to \mathfrak{sl}_2(\mathbb{F}).\]
In concrete terms, any element in $P_i$ (resp. $\mathfrak{p}_i$) is a blockwise upper triangular matrix with one block $X$ of size 2 and the other blocks of size 1, and the map $\pi_i$ (resp. $d\pi_i$) is obtained by considering the projection of $X$ to $\PGL_2(\bbF)$ (resp. $\mathfrak{sl}_2(\bbF)$).
In Appendix {\upshape\S\ref{appendix:calculation.stabilzer}},  
we have calculated the image of $P_i^\tau$ (in fact, of $\mathfrak{p}_i^\tau=\mathrm{Lie}(P_i^\tau)$) by $\pi_i$ (in fact, $d\pi_i$). 
The calculations show the following alternative:
\begin{itemize}
\item[\rm (A)] $s_i\tau=\tau$, in which case $d\pi_i(\mathfrak{p}_i^\tau)=\mathfrak{sl}_2(\bbF)$;
\item[\rm (B)] $s_i\tau\not=\tau$, in which case $d\pi_i(\mathfrak{p}_i^\tau)$ is a Borel subalgebra of $\mathfrak{sl}_2(\bbF)$.
\end{itemize}
One can be more precise. There are in fact three cases. Here we refer to $i,i+1$ as vertices in the graphic representation of $\tau$.
\begin{itemize}
\item[\rm (I)] If $i,i+1$ are both of degree 0 or both of degree 2, then we are in case (A).
\item[\rm (II)] If $\deg_\tau(i)<\deg_\tau(i+1)$ or $i,i+1$ are end points of two edges which have a crossing, then we are in case (B) and, moreover, $d\pi_i(\mathfrak{p}_i^\tau)$ is the subalgebra of \emph{lower} triangular matrices in $\mathfrak{sl}_2(\mathbb{F})$;
\item[\rm (III)] If $\deg_\tau(i)>\deg_\tau(i+1)$ or $i,i+1$ are end points of two edges which do not have a crossing, then we are in case (B) and, moreover, $d\pi_i(\mathfrak{p}_i^\tau)$ is the subalgebra of \emph{upper} triangular matrices in $\mathfrak{sl}_2(\mathbb{F})$.  
\end{itemize}
In the language of Knop's paper \cite[\S3]{Knop.CMH1995}:
\begin{itemize}
\item In case (A), $\Phi(P_i)$ is of type $G_0$;
\item In case (B), $\Phi(P_i)$ is of type $S\cdot U_0$.
\end{itemize}
Types $T_0$ and $N_0$ of 
\cite[\S3]{Knop.CMH1995}
do not appear in our situation.  
We can check this if we consider the type of the stabilizer and consider the claims just after 
\cite[Lemma 3.1]{Knop.CMH1995}.

In particular,
the information on isotropy subgroups/subalgebras can be used in combination with 
\cite[table on p.~295]{Knop.CMH1995} in order to determine $B_K s_i B_K\cdot \tau$.
First, we note that $P_i\cdot\tau$ always contains the orbits $B_K\cdot\tau$ and $B_K\cdot(s_i\tau)$, which can be the same. In case (B),  where $\Phi(P_i)$ is of type $S\cdot U_0$, we also know from \cite{Knop.CMH1995} that $P_i\cdot\tau$ contains exactly two orbits, namely
\[P_i\cdot\tau =B_K\cdot \tau\cup B_K\cdot s_i\tau.\] 
In this case, $B_K s_i B_K\cdot\tau$ contains at most two orbits, hence we have either
$B_K s_i B_K\cdot\tau=B_K\cdot s_i\tau$ or $B_K s_i B_K\cdot\tau=B_K\cdot\tau\cup B_K\cdot s_i\tau$.
It remains to determine in which case we have indeed two orbits.

Let $B_K=TU$ where $T$ is the standard maximal torus and $U\subset B_K$ is the unipotent radical.
Let $X_i^\pm:=\{u_i^\pm(t)\}_{t\in\mathbb{F}}$ be the one parameter subgroup of unipotent matrices attached to the root $\pm\alpha_i$. Thus
\[U=U_i X_i^+\quad\mbox{and}\quad s_i X_i^+ s_i^{-1}=X_i^-.\]
Hence
\[B_K s_i B_K=B_K s_i U_i X_i^+=B_K s_i X_i^+=B_K X_i^- s_i.\]
Whence
\[B_K s_i B_K\cdot\tau=B_K X_i^-\cdot s_i\tau.\]
\begin{itemize}
\item 
In case (I), we have $s_i\tau=\tau$ and $ X_i^- \in P_i^{\tau} $, 
hence $B_K s_i B_K\cdot\tau=B_K\cdot\tau$ in this case.
\item 
In case (II), we have that the projection $\pi_i(P_i^{s_i\tau})$ of the isotropy group of $s_i\tau$ consists of upper triangular matrices (since by applying $s_i$ to $\tau$, we switch the configuration of the vertices $i,i+1$). 
This means $X_i^-\not\subset P_i^{s_i\tau}$, and 
hence there exists $g\in X_i^-$ such that $g\cdot s_i\tau\not= s_i\tau$. 
We claim that $g\cdot s_i\tau\notin B_K\cdot s_i\tau$, 
so that $g\cdot s_i\tau\in B_K\cdot \tau$ and we must have $B_K s_i B_K\cdot \tau=B_K\cdot\tau\cup B_K\cdot s_i\tau$ (as asserted in the lemma) in this case.
Arguing by contradiction assume that $g\cdot s_i\tau=b \cdot s_i\tau$ with some $b\in B_K$. 
Then $g^{-1}b\in P_i^{s_i\tau}$, which implies that $\pi_i(g^{-1}b)$ must be upper triangular. But this is not the case, hence the claim is verified.
\item In case (III), the projection $\pi_i(P_i^{s_i\tau})$ of the isotropy group of $s_i\tau$ consists of lower triangular matrices. 
Hence $X_i^-\subset P_i^{s_i\tau}$. Whence $B_K s_i B_K\cdot \tau=B_KX_i^-\cdot s_i\tau=B_K\cdot s_i\tau$ in this case (as asserted in the lemma).
\end{itemize}
The proof of Lemma \ref{key-lemma:2coset.multiplication} is complete for $ i $ associated to $ \symmetricgrpof{p} $, 
i.e., $ 0 < i < p $.  
The case for $ s_i \in \symmetricgrpof{q} $ can be argued similarly.  
\end{proof}

\section{Explicit action of Hecke algebra on the double flag variety}

In this section, $ \bbF $ is a finite field of characteristic $\neq 2 $ again.  
As before we denote $ \q = \# \bbF $.
Note that Lemma \ref{key-lemma:2coset.multiplication} 
is still valid in this context (by considering fixed points of the Frobenius map).

Recall that the Hecke algebra $ \Hecke = \Hecke(K, B_K) $ acts on the space of
$ K $-orbits $ \C\,\Xfv/K $ and the action is given by 
the general theory of spherical functions discussed in \S\ref{section:Hecke.alg.2.cosets}.

According to the theory, 
by Theorem \ref{thm:action.of.fw.generalsetting} and Lemma \ref{key-lemma:2coset.multiplication}, 
we get 
\begin{equation*}
T_i \ast \xi_{\tau} = \alpha \xi_{\tau} + \beta \xi_{s_i \tau}
\end{equation*}
for some coefficients $ \alpha, \beta \in \Q $  (one of which might be zero).  
Let us determine them.

\subsection{Calculation of $ \alpha $}

Note that $ \alpha \neq 0 $ only if we are in Cases (I) or (II) in 
Lemma \ref{key-lemma:2coset.multiplication}.

To compute it, we use the formula in Theorem~\ref{thm:action.of.fw.generalsetting} with $ a_i = e $ (identity).  
The numerator becomes (before counting the number)
\begin{equation*}
K_{\tau} B_K \cap B_K s_i B_K, 
\quad
\text{ where } K_{\tau} = K \cap P_{[\tau]},  
\end{equation*}
and $ P_{[\tau]} = \Stab_G([\tau]) $ is the stabilizer of the $ r $-dimensional space 
$ [\tau] \in \Grass_r(\bbF^n) $ generated by the columns of $ \tau $.  
From a general argument, 
\begin{equation*}
B_K s_i B_K = X_i^+ s_i B_K 
\qquad 
\text{ with }
\qquad
X_i^+ = U_{\alpha_i} \simeq \bbF
\end{equation*}
where $ U_{\alpha_i} \subset B_K $ 
denotes the one parameter subgroup generated by a root vector $ x_{\alpha_i} $ corresponding to $ s_i = s_{\alpha_i} $.  

\begin{lemma}\label{lemma:KtauBK}
\begin{equation*}
K_{\tau} B_K \cap B_K s_i B_K 
= \{ u s_i b \in X_i^+ s_i B_K \mid s_i u^{-1} \tau \in B_K \tau \}.
\end{equation*}
The expression $ u s_i b $ is unique.
\end{lemma}

\begin{proof}
Write $ u s_i b \in B_K s_i B_K = X_i^+ s_i B_K $ for $ u \in X_i^+ $ and $ b \in B_K $.  
\begin{equation*}
u s_i b \in K_{\tau} B_K \iff u s_i \in K_{\tau} B_K 
\iff (u s_i)^{-1} \in B_K K_{\tau} 
\iff s_i u^{-1} \tau \in B_K \tau.
\end{equation*}
\end{proof}

\begin{lemma}\label{lemma:alpha.caseI}
Assume we are in Case (I) so that $ s_i \tau = \tau $.  
Then 
\begin{equation*}
K_{\tau} B_K \cap B_K s_i B_K 
= X_i^+ s_i B_K \simeq \bbF \times B_K .
\end{equation*}
This means 
$ \alpha = \# \bbF = \q $.  
\end{lemma}

\begin{proof}
We will apply Lemma~\ref{lemma:KtauBK}.
Since $ s_i \tau = \tau $, 
if we denote by $ v = s_i u s_i $, a generator of the one parameter subgroup corresponding to the negative root $ -\alpha_i $, 
we get 
\begin{equation*}
s_i u^{-1} \tau = (s_i u s_i)^{-1} s_i \tau = v^{-1} \tau 
\end{equation*}
and 
$ v^{-1} \tau \in B_K \tau $ holds for any $ v $ (according to Lemma \ref{lemma:image.of.parabolic}\eqref{lemma:image.of.parabolic:item:1}).  
Thus $ u \in X_i^+ $ is arbitrary.
\end{proof}

\begin{lemma}
Assume we are in Case (II) so that $ s_i \tau \neq \tau $.  
Then 
\begin{equation*}
K_{\tau} B_K \cap B_K s_i B_K 
= (X_i^+ \setminus \{ e \}) s_i B_K \simeq \bbF^{\times} \times B_K .
\end{equation*}
This means 
$ \alpha = \# \bbF - 1 = \q - 1 $.  
\end{lemma}

\begin{proof}
As in the proof of Lemma \ref{lemma:alpha.caseI}, we denote $ v = s_i u s_i $.  
We get 
\begin{equation*}
s_i u^{-1} \tau = (s_i u s_i)^{-1} s_i \tau = v^{-1} s_i \tau .
\end{equation*}
Since $ s_i \tau \neq \tau $, 
this is in $ B_K \tau $ iff 
$ v^{-1} s_i \tau \not\in B_K s_i \tau $, 
iff $ v \neq e $ 
(this follows from a similar arguments as in the end of the proof of Lemma~\ref{key-lemma:2coset.multiplication}.
See Lemma \ref{lemma:image.of.parabolic}\eqref{lemma:image.of.parabolic:item:2} also).  
This proves the lemma.
\end{proof}

\subsection{Calculation of $ \beta $}

The case $ \beta \neq 0 $ only occurs for the cases (II) and (III) in Lemma \ref{key-lemma:2coset.multiplication}.  
Thus we can assume $ s_i \tau \neq \tau $.  

To compute $ \beta $, as in the case of $ \alpha $, 
we use the formula in Theorem \ref{thm:action.of.fw.generalsetting} with $ a_i = s_i $.  
The numerator becomes (before counting the number)
\begin{equation*}
s_i K_{\tau} B_K \cap B_K s_i B_K 
= s_i K_{\tau} B_K \cap X_i^+ s_i B_K .
\end{equation*}
Let us denote $ X_i^- = s_i X_i^+ s_i $.  
Thus, we need to compute the number of elements in 
\begin{equation*}
K_{\tau} B_K \cap s_i X_i^+ s_i B_K = K_{\tau} B_K \cap X_i^- B_K .
\end{equation*}

\begin{lemma}
\label{lemma:calculation.of.beta}
\begin{equation*}
\beta = \# \{ v \in X_i^- \mid v \tau \in B_K \tau \} 
= \begin{cases}
\q & \text{ if $ \tau $ is in Case (II), }
\\
1 & \text{ if $ \tau $ is in Case (III). }
\end{cases}
\end{equation*}
\end{lemma}

\begin{proof}
Let $V=\{v\in X_i^-\mid v\in K_\tau B_K\}$. Note that the mapping
\[V\times B_K\to K_\tau B_K\cap X_i^-B_K,\ (v,b)\mapsto vb\]
is bijective. 
(It is clearly well defined.)
It is injective since, if $vb=v'b'$ for $vv'\in V$ and $bb'\in B_K$, then we get $v'^{-1}v=b'b^{-1}\in X_i^-\cap B_K=\{e\}$ hence $(v,b)=(v',b')$.
It is surjective since any element in $K_\tau B_K\cap X_i^- B_K$ can be written $vb$ with $v\in X_i^-$ and $b\in B_K$, and we have $v=(vb)b^{-1}\in K_\tau B_K$, hence $v\in V$. This observation combined with 
Theorem \ref{thm:action.of.fw.generalsetting} (and the discussion above the statement of this lemma) implies that 
\[\beta=\frac{\#(V\times B_K)}{\# B_K}=\# V.\]

Next, for $v\in X_i^-$, we note that
\[v\in V\Leftrightarrow v\in K_\tau B_K\Leftrightarrow v^{-1}\in B_KK_\tau\Leftrightarrow v^{-1}\tau\in B_K\tau.\]
This yields a well-defined bijection $V\to \{v\in X_i^-\mid v\tau\in B_K\tau\}$, $v\mapsto v^{-1}$. Hence
\[\beta=\#\{v\in X_i^-\mid v\tau\in B_K\tau\}\]
as asserted in the lemma.

It remains to show the second equality in Lemma \ref{lemma:calculation.of.beta}.
First assume that $\tau$ is in Case (II).  
In this case, as recalled in Section \ref{section:2.cosets.multiplication} we have $X_i^-\subset K_\tau$, hence $v\tau=\tau\in B_K\tau$ for all $v\in X_i^-$. This implies that $\{v\in X_i^-\mid v\tau\in B_K\tau\}=X_i^-$, hence $ \beta = \#X_i^- = \q $ in this case.

Finally assume  that $\tau$ is in Case (III). 
In this case, we claim that $\{v\in X_i^-\mid v\tau\in B_K\tau\}=\{e\}$, and this will imply that $\beta=1$ as asserted. Thus it remains to establish the claim. To this end, let $v\in X_i^-$ be such that $v\tau\in B_K\tau$. Let us write $v\tau=b\tau$ with $b\in B_K$. This implies that $v^{-1}b\in P_i^\tau$ where $P_i=B_K s_i B_K\sqcup B_K$ and $P_i^\tau=P_i\cap K_\tau$ (notation of Section \ref{section:2.cosets.multiplication}). Hence $\pi_i(v^{-1}b)\in \pi_i(P_i^\tau)$ (where, as in Section \ref{section:2.cosets.multiplication}, $\pi_i$ denotes the projection to the  $(i,i+1)$-block). As used in Section \ref{section:2.cosets.multiplication}, the fact that $\tau$ is in Case (III) implies that $\pi_i(P_i^\tau)$ is formed by upper-triangular matrices. But $\pi_i(v^{-1}b)$ is upper triangular if and only if $v=e$. Whence $v=e$, and the claim is established.
\end{proof}

\subsection{Action of simple reflections}

Let us recall Cases (I)--(III) from \S \ref{section:2.cosets.multiplication}.  

\begin{theorem}\label{thm:Hecke.algebra.action.generators}
The Hecke algebra $ \Hecke = \Hecke(K, B_K) $ acts on the space of 
$ K $-orbits $ \C\,\Xfv/K $ and the action is explicitly given 
by the formula:
\begin{equation}
T_i \ast \xi_{\tau} 
= \begin{cases}
\q \xi_{\tau} & (s_i \tau = \tau) \text{ in Case \upshape(I)} ,
\\
(\q - 1) \xi_{\tau} + \q \xi_{s_i \tau} & (s_i \tau \neq \tau) \text{ in Case \upshape(II)} ,
\\
\xi_{s_i \tau}  &       (s_i \tau \neq \tau) \text{ in Case \upshape(III)} ,
\end{cases}
\end{equation}
where $ \{ T_i \} $ are the generators of $ \Hecke $ corresponding to 
the simple reflections.
\end{theorem}

Note that 
the Borel-Moore homology 
of the conormal variety $ \conormal $ 
has its basis consisting of the closures of conormal bundles of $ K $-orbits on $ \Xfv $.  
So the above theorem tells that the space of top Borel-Moore    homology 
has a natural Hecke module structure.

\skipover{
This theorem requires no proof (already proved).  
We only check the Hecke algebra relation \eqref{eq:Hecke.algebra.R1} below in order to convince ourselves.
\begin{align}
&
(T_s + 1) (T_s - \q) = 0 
\label{eq:Hecke.algebra.R1}
\\
&
T_{w w'} = T_w T_{w'} \qquad \text{ if $ \lengthof{ww'} = \lengthof{w} + \lengthof{w'} $}
\label{eq:Hecke.algebra.R2}
\end{align}
We will write $ T_s $ for $ T_i $ ($ s = s_i $).  

First, let us assume $ s \tau = \tau $.  
Then 
\begin{equation*}
(T_s - \q) \xi_{\tau} = 0
\end{equation*}
and \eqref{eq:Hecke.algebra.R1} holds.

Second, let us assume $ s \tau \neq \tau $ and we are in Case (II).  
In this case, note that $ \tau' = s \tau $ is in Case (III).
\footnote{I believe so, but we need to check.}
Then we calculate
\begin{align}
(T_s + 1)(T_s - \q) \xi_{\tau}
&= (T_s + 1) \bigl( (\q - 1) \xi_{\tau} + \q \xi_{s \tau} - \q \xi_{\tau} \bigr)
\\
&= (T_s + 1) \bigl( - \xi_{\tau} + \q \xi_{s \tau} \bigr)
\\
&= - \bigl( (\q - 1) \xi_{\tau} + \q \xi_{s \tau} + \xi_{\tau} \bigr) + \q (\xi_{\tau} + \xi_{s \tau})
= 0
\end{align}
Similarly, if we are in Case (III) so that $ \tau' = s \tau $ is in Case (II), 
\begin{align}
(T_s + 1)(T_s - \q) \xi_{\tau}
&= (T_s + 1) \bigl( \xi_{s \tau} - \q \xi_{\tau} \bigr)
\\
&= \bigl( (\q - 1) \xi_{s \tau} + \q \xi_{\tau} + \xi_{s \tau} \bigr) 
- \q \bigl( \xi_{s \tau} + \xi_{\tau} \bigr) = 0
\end{align}
so that \eqref{eq:Hecke.algebra.R1} holds.
}

\section{Representation of the Weyl group}

We get the action of Hecke algebra in terms of generators $ T_{s_i} $'s.  
If we specialize the action by putting $ \q = 1 $, then we get an action of 
the Weyl group $ W_K = \symmetricgrpof{p} \times \symmetricgrpof{q} $.  

From Theorem \ref{thm:Hecke.algebra.action.generators}, 
a simple reflection $ s_i \in W_K $ acts on $ \tau $ simply by the multiplication 
$ s_i \tau $, which causes the transposition of $ i $-th and $ (i +1) $-th rows of the $ (p + q) \times r $-matrix $ \tau $.  
So the action of the Weyl group on $ \tau $ is simply by the multiplication of permutation matrices from the left on 
the space of partial permutations.  

In the graphical notation of $ \tau $, 
$ w \in W_K $ acts on $ \tau $ as a permutation of $ \Vertices^+_p \times \Vertices^-_q $.  
Thus we can easily see what kind of representations of $ W_K $ we get.

\begin{theorem}\label{thm:Weyl.group.representation}
The Weyl group $ W_K = \symmetricgrpof{p} \times \symmetricgrpof{q} $ acts on the orbit space $ \C\,\Xfv/K $, 
and we have the following equivalence as representations of $ W_K $.
\begin{equation*}
\C\,\Xfv/K
\simeq \bigoplus_{(k,s,t)} \Ind_{H_{k,s,t}}^{\permutationsof{p}\times\permutationsof{q}} \trivial , 
\end{equation*}
where the sums are running over triples $(k,s,t)$ given in \S\ref{subsection:number.of.orbits}, 
and the subgroup 
$ H_{k,s,t} = \Delta\permutationsof{k} {\times}  \permutationsof{s} {\times}  \permutationsof{s'} {\times}  \permutationsof{t} {\times}  \permutationsof{t'} $ 
is defined in the same place.
\end{theorem}

Since the dimension of the representation coincides with the number of orbits, we retrieve the formula of the number of orbits 
(Theorem \ref{thm:number.orbits}).

\section{Appendix: Calculation of the stabilizer}\label{appendix:calculation.stabilzer}

\newcommand{\og}{\overline{g}}

Let $ \tau \in \parameters $ and we consider the orbit 
$ B_K \cdot [\tau] \subset \Grass_r(V) $.  
Let $ P_{\alpha} \subset K $ be a standard minimal parabolic subgroup 
associated to a simple root $ \alpha $.  
Then $ P_{\alpha}/B_K $ can be identified with $ \bbP^1 $, 
in fact $ B_K = \Stab(\flag^+_0, \flag^-_0) $, 
where $ \flag^{\pm}_0 $ are the standard flags of $ V^{\pm} $ respectively.  

Let us follow the notation of Bourbaki for root systems (\cite{Bourbaki.LieChap4-6}).
In our case, the root system of $ K $ is $ A_{p-1} + A_{q -1} $, 
and thus $ \alpha = \alpha_i = \ee_i - \ee_{i + 1} $ ($ 0 < i < p $ or $ p < i < p + q $).

If $ \alpha = \alpha_i \;\; (0 < i < p) $ then 
writing 
$ \flag^+_0 = (F^+_{0, 0} , \dots, F^+_{0, p}) $, 
we have:
\begin{equation*}
P_{\alpha}/B_K 
\simeq \{ W \mid F^+_{0, {i} - 1} \subset W \subset F^+_{0, {i} + 1} \}
\simeq \bbP(F^+_{0, {i} + 1}/ F^+_{0, {i} - 1}) \simeq \bbP^1 .
\end{equation*}
Thus we conclude 
\begin{equation*}
\Aut(P_{\alpha}/B_K) = \PGL(V_{\alpha}), \qquad
\text{where } \;\;
V_{\alpha} := F^+_{0, {i} + 1}/ F^+_{0, {i} - 1}.
\end{equation*}
Any element $ g \in P_{\alpha} $ determines $ \og \in \PGL(V_{\alpha}) $.  

Let us write this more precisely.  
We have a Levi decomposition 
$ P_{\alpha} = L_{\alpha} U_{\alpha} $, 
where 
$ U_{\alpha} $ denotes the unipotent radical, and 
$ L_{\alpha} $ is the standard Levi subgroup 
isomorphic to $ \GL(V_{\alpha}) \times \bbG_m^{p - 2} \times \bbG_m^q $.  
Thus any $ g \in P_{\alpha} $ can be written in the form 
\begin{equation*}
g = (g_{\alpha}, t_1, t_2 ) \cdot u 
\in (\GL(V_{\alpha}) \times \bbG_m^{p - 2} \times \bbG_m^q) \ltimes U_{\alpha} .
\end{equation*}
We define 
$ \varphi_{\alpha}(g) = g_{\alpha} \in \GL(V_{\alpha}) $, 
the projection to the $ \GL(V_{\alpha}) $-component.  

For $ \tau \in \parameters $, 
we have to consider the stabilizer 
$ P_{\alpha}^{\tau} $ of $ [\tau] $ in $ P_{\alpha} $ and 
its Lie algebra $ \lie{p}_{\alpha}^{\tau} $,  
and their images by $ \varphi_{\alpha} $ and 
$ d\varphi_{\alpha} $ respectively.

\begin{lemma}\label{lemma:image.of.parabolic}
\begin{penumerate}
\item\label{lemma:image.of.parabolic:item:1}
If $ s_{\alpha} \tau = \tau $ then 
$ d\varphi_{\alpha}(\lie{p}_{\alpha}^{\tau}) = \lie{gl}_2 $ holds.
\item\label{lemma:image.of.parabolic:item:2}
If $ s_{\alpha} \tau \neq \tau $ then 
$ d\varphi_{\alpha}(\lie{p}_{\alpha}^{\tau}) $ is a Borel subalgebra of $ \lie{gl}_2 $.
\end{penumerate}
\end{lemma}

\begin{proof}
We use the notation of \S\ref{S3.3}.

\eqref{lemma:image.of.parabolic:item:1}
If $ s_{\alpha} \tau = \tau $ then either 
the vertices $ {i} $ and $ {i} + 1 $ are both in 
the set $ L' $ of unmarked vertices; 
or $ {i} $ and $ {i} + 1 $ are both 
marked belonging to the set $ L $.  

In the first case, we have 
$ [\tau] \subset \langle \eb_s^+ \mid s \not\in 
\{ {i}, {i} + 1 \} \rangle \oplus V^- $.  
In the second case, 
$ \langle \eb_{i}^+, \eb_{i + 1}^+ \rangle \subset [\tau] $.  
In both cases, 
for any $ h \in \GL_2 $, 
$ g := \diag (1, \dots, 1, h, 1, \dots, 1 ) 
\in P_{\alpha}^{\tau} $, 
where 
$ h $ appears in the diagonal block of 
$ {i} $-th and $ ({i} + 1) $-th rows.  
Whence 
$ \varphi_{\alpha}(P_{\alpha}^{\tau}) = \GL_2 $ 
in this case.

\eqref{lemma:image.of.parabolic:item:2}
Assume $ s_{\alpha} \tau \neq \tau $.  
A general description of 
the Lie algebra of the stabilizer tells 
\begin{equation*}
\lie{k}^{\tau} = \{ x \in \lie{k} \mid x([\tau]) \subset [\tau] \} .
\end{equation*}
Write 
$ x \in \lie{k} $ as 
$ x = \diag (x^+ , x^-) $ and 
\begin{equation*}
\begin{aligned}{}
[\tau] 
&= \langle \eb_s^+ \mid \text{$ s^+ $ is marked} \rangle 
\oplus \langle \eb_t^- \mid \text{$ t^- $ is marked} \rangle 
\\
&\qquad
\oplus \langle \eb_s^+ + \eb_t^- \mid 
\text{there is an edge $ (s^+, t^-) $} \rangle . 
\end{aligned}
\end{equation*}
Note the followings hold for $ x \in \lie{k}^{\tau} $.
\begin{itemize}
\item
If $ s \in L' $ and $ t \in M' $, then we have 
$ x^+_{s, k} = 0 \;\; (k \in L \cup I) $ and 
$ x^-_{t, \ell} = 0 \;\; (\ell \in M \cup J) $. 
\item
If $ s \in I $ and $ t \in J $, then we have 
$ x^+_{s, k} = 0 \;\; (k \in L) $ and 
$ x^-_{t, \ell} = 0 \;\; (\ell \in M) $. 
\item
If $ (s^+, t^-) $ and $ (k^+, \ell^-) $ are edges, then we have 
$ x^+_{s, k} = x^-_{t, \ell} $.
\end{itemize}
In fact these conditions exactly characterizes the stabilizer $ \lie{k}^{\tau} $. 

Based on these conditions, we can compute 
$ d\varphi_{\alpha}(\lie{p}_{\alpha}^{\tau}) $ explicitly.  
We divide the cases into eight, and examine each case.  
These eight cases are listed in Figure \ref{figure:alpha.component.stabilizer.Borel} below, 
where we denote the upper/lower triangular Borel subalgebras by $ \lie{b}_2^{\pm} $.

\begin{figure}[htbp]\label{figure:alpha.component.stabilizer.Borel}
\caption{Table: $ \alpha $-component of the stabilizer in Case (2), where $ \alpha = \ee_i - \ee_{i + 1} $.}
\newcommand{\isolatedvertex}{\text{\tiny$\bullet$}}
\newcommand{\markedvertex}{\makebox[0pt][l]{\,\tiny$\bullet$}\scriptstyle\bigcirc}
\newcommand{\edgedvertices}{\makebox[0pt][c]{\hspace*{4ex}\graphQ}}
\newcommand{\tinyvstrut}{\vphantom{\big|}}
\begin{equation*}
\begin{array}{c|cc|c}
 & \text{condition} & \text{graphical notation} & d\varphi_{\alpha}(\lie{p}_{\alpha}^{\tau}) \\
\hline
(1) & {i} \in L', {i} + 1 \in L &
\xymatrix@C -4ex{\overset{{i}}{{\isolatedvertex}\tinyvstrut} & \overset{{i} + 1}{{\markedvertex}\tinyvstrut}} 
& \lie{b}_2^- \\ \hline
(2) & {i} \in L, {i} + 1 \in L' &
\xymatrix@C -4ex{\overset{{i}}{{\markedvertex}\tinyvstrut} & \overset{{i} + 1}{{\isolatedvertex}\tinyvstrut}} 
& \lie{b}_2^+ \\ \hline
(3) & {i} \in L', {i} + 1 \in I &
\xymatrix@C -4ex{\overset{{i}}{{\isolatedvertex}\tinyvstrut} & \overset{{i} + 1}{\raisebox{-3.7ex}[1ex][3ex]{\edgedvertices}\tinyvstrut}} 
& \lie{b}_2^- \\[4ex] \hline
(4) & {i} \in I, {i} + 1 \in L' & 
\xymatrix@C -4ex{\overset{{i}}{\raisebox{-3.7ex}[1ex][3ex]{\edgedvertices}\tinyvstrut} & \overset{{i} + 1}{{\isolatedvertex}\tinyvstrut}} 
& \lie{b}_2^+ \\[4ex] \hline
(5) & {i} \in L, {i} + 1 \in I &
\xymatrix@C -4ex{\overset{{i}}{{\markedvertex}\tinyvstrut} & \overset{{i} + 1}{\raisebox{-3.7ex}[1ex][3ex]{\edgedvertices}\tinyvstrut}} 
& \lie{b}_2^+ \\[4ex] \hline
(6) & {i} \in I, {i} + 1 \in L & 
\xymatrix@C -4ex{\overset{{i}}{\raisebox{-3.7ex}[1ex][3ex]{\edgedvertices}\tinyvstrut} & \overset{{i} + 1}{{\markedvertex}\tinyvstrut}} 
& \lie{b}_2^- \\[4ex] \hline
(7) & \begin{array}{c}
{i}, {i} + 1 \in I, \; k < \ell \\
({i}, k) \text{ and } ({i} + 1, \ell) \text{ are edges}
\end{array}
& 
\begin{smallmatrix}
{i} &  & {i} + 1 \vphantom{\big|}
\\[.5ex]
\edgedvertices & & \edgedvertices
\\
k & < & \ell \vphantom{\big|}
\end{smallmatrix}
& \lie{b}_2^+ \\[4ex] \hline
(8) & \begin{array}{c}
{i}, {i} + 1 \in I, \; \ell < k \\
({i}, k) \text{ and } ({i} + 1, \ell) \text{ are edges}
\end{array}
& 
\begin{smallmatrix}
{i} &  & {i} + 1 \vphantom{\big|}
\\[.5ex]
& \makebox[0pt][c]{\graphA} &
\\[0ex]
\ell & < & k\vphantom{\big|}
\end{smallmatrix}
& \lie{b}_2^- \\[4ex] \hline
\end{array}
\end{equation*}
Note that in Cases (7) and (8), 
we must have $ x^+_{{i}, {i}+1} = x^-_{k, \ell} $ and 
$ x^+_{{i} + 1, {i}} = x^-_{\ell, k} $, respectively.  
Moreover, in Case (7), we have $ x^-_{\ell, k} = 0 $, and 
in Case (8), $ x^-_{k, \ell} = 0 $.
\end{figure}
\end{proof}

\clearpage


\def\cftil#1{\ifmmode\setbox7\hbox{$\accent"5E#1$}\else
  \setbox7\hbox{\accent"5E#1}\penalty 10000\relax\fi\raise 1\ht7
  \hbox{\lower1.15ex\hbox to 1\wd7{\hss\accent"7E\hss}}\penalty 10000
  \hskip-1\wd7\penalty 10000\box7} \def\cprime{$'$} \def\cprime{$'$}
  \def\Dbar{\leavevmode\lower.6ex\hbox to 0pt{\hskip-.23ex \accent"16\hss}D}
\providecommand{\bysame}{\leavevmode\hbox to3em{\hrulefill}\thinspace}
\providecommand{\MR}{\relax\ifhmode\unskip\space\fi MR }
\providecommand{\MRhref}[2]{%
  \href{http://www.ams.org/mathscinet-getitem?mr=#1}{#2}
}
\providecommand{\href}[2]{#2}

\end{document}